\documentclass[11pt]{article}
\usepackage[letterpaper,margin=1in]{geometry}
\usepackage{todonotes}
\usepackage[T1]{fontenc}
\usepackage{lmodern}
\usepackage{amsmath,amssymb,amsthm,mathtools,bbm}
\usepackage{microtype}
\usepackage[hidelinks]{hyperref}
\usepackage{enumitem}
\usepackage{fancyhdr}
\newtheorem{theorem}{Theorem}
\newtheorem{lemma}[theorem]{Lemma}

\theoremstyle{remark}
\newtheorem{remark}[theorem]{Remark}
\newtheorem{definition}[theorem]{Definition}
\newtheorem{example}[theorem]{Example}
\newcommand{\Z}{\mathbb Z}
\newcommand{\R}{\mathbb R}
\newcommand{\Q}{\mathbb Q}
\newcommand{\E}{\mathbb E}
\newcommand{\Prob}{\mathbb P}
\renewcommand{\P}{\mathbb P}
\newcommand{\ee}{\varepsilon}
\newcommand{\norm}[1]{\left\lVert #1\right\rVert}
\newcommand{\bs}[1]{\boldsymbol{#1}}
\newcommand{\sss}[1]{\scriptscriptstyle{#1}}
\newcommand{\mc}[1]{\mathcal{#1}}

\newcommand{\sone}{\sss{\{1\}}}
\newcommand{\stwo}{\sss{\{2\}}}
\newcommand{\sr}{\sss{\{r\}}}

\newcommand{\START}{$\blacktriangleright\,$ }
\newcommand{\END}{\hfill $\,\blacktriangleleft$}
\title{The velocity is deterministic in random walk in 2-valued i.i.d.~random environment.}
\author{Mark~Holmes\thanks{University of Melbourne}  ~and Thomas~S.~Salisbury\thanks{York University}}
\date{October 20, 2026}
\begin{document}
\maketitle
\thispagestyle{plain}

\begin{abstract}
We prove that the limiting velocity for (nearest-neighbour) random walk in i.i.d.~random environment is deterministic when the local environment is supported on two points.  No ellipticity is assumed.  The proof uses 
a clever adaption (suggested by ChatGPT) of coupling arguments previously developed by the authors.  It extends to i.i.d.~environments where the local environments are mixtures of two such points.  
\end{abstract}

\section{Introduction and Main results}

Fix $d\ge 1$ and let $\mathcal E=\{\pm e_1,\ldots,\pm e_d\}$.  An \emph{environment} $\bs{\omega}=(\omega_x)_{x\in \Z^d}$ is a family of probability distributions (more precisely, probability mass functions) on $\mc{E}$, indexed by sites of the lattice $\Z^d$.  In other words, for every $x\in \Z^d$, $\omega_x(e)\ge 0$ for every $e\in \mc{E}$ and $\sum_{e\in \mc{E}}\omega_x(e)=1$.  These probability distributions can be used to define the law of a discrete-time, time-homogeneous Markov chain $\bs{X}=(X_n)_{n\ge 0}$ on $\Z^d$.  Starting e.g. from $X_0=0$, when the walker is at $x$ at time $n$, the probability that it steps to $x+e$ at the next time step is $\omega_x(e)$.

When $\bs{\omega}$ is itself chosen randomly, according to some measure $\mathbb{Q}$, one can ask questions about the behaviour of the walk $(X_n)_{n \ge 0}$ for a fixed  environment, or about the joint law $\P$ of the environment and the walk.  Under $\P$ if we just ask about the environment this is equivalent to $\Q$, while if we just ask about the random walk path it corresponds to averaging  over the possible environments, and the random walk is not a Markov chain (as its history reveals information about the environment).   Among the most important and well-studied settings is that of \emph{random walk in i.i.d.~random environment} when the environment is chosen to be i.i.d.~over the sites.  Henceforth we consider only this setting.

 Following the fundamental work of Solomon \cite{Solomon}, the 1-dimensional case is rather well understood.  In particular, there is a deterministic limiting velocity $v=\lim_{n\to \infty}n^{-1}X_n$ (i.e.~it is the same for $\mathbb{P}$-almost every environment).  This is also known to hold in 2 dimensions.  More generally it was known that in $\Z^d$, $n^{-1}X_n$ converges almost surely for $\P$-almost every environment but that the limit could be random.    Proving that the limit is deterministic has been one of the major open problems in the field in the last 50 years.    In this paper we prove this result when the local  random environment can only take 2 possible values.  We do not assume that the walk is \textit{elliptic} (i.e.~that all $\omega_x(e)>0$).  The proof uses an adaption of a coupling argument of \cite{HS} that was suggested by ChatGPT.  The method generalises to i.i.d.~environments that are random mixtures of two fixed environments.  
 
 Around the same time that this paper was being written (October 2026), it was announced (as one of hundreds of  new results) by OpenAI that they had a  proof of the full result in the case of  ellipticity.  At the time of writing we are not aware that this proof has been verified by a human.

An (i.i.d.) environment is \emph{$2$-valued} if there exist two distinct probability distributions (probability mass functions to be precise) $\mu_1,\mu_2$ on $\mathcal E$, and a parameter $p$ such that $(\omega_x)_{x\in \Z^d}$ are i.i.d.~with 
\[\Q(\omega_o=\mu_1)=p=1-\Q(\omega_o=\mu_2).\]
\begin{example}
\START
Let $\mu_1$ be the uniform distribution on $\mc{E}_+=\{e_1,\dots, e_d\}$ and $\mu_2$ be the uniform distribution on $\mc{E}_-=\{-e_1,\dots, -e_d\}$.  This is a non-elliptic environment (called the orthant model) for which the range of the random walk is infinite and the velocity is believed to be non-zero for $p\ne 1/2$.  The velocity is known to be  deterministic (and non-zero) if $p$ is sufficiently large or sufficiently small \cite{HSRWDRE2}.   \footnote{Other methods can be used to show that the velocity is deterministic in high dimensions}\END
\end{example}
Set
\[
 m_r=\sum_{e\in\mathcal E}e\,\mu_r(e),\qquad r=1,2.
\]
For a realisation $\bs{\omega}$ of an i.i.d.~environment, let $P^{\bs{\omega}}$ be the law of the discrete-time, time homogenous Markov chain $\bs{X}$ on $\Z^d$ with transition probabilities at $x$ given by $\omega_x$, i.e.
\[
 P^{\bs{\omega}}(X_{n+1}=x+e\mid X_n=x)=\omega_x(e).
\]
 Henceforth we assume that $d\ge 2$ is fixed.  Our main result is the following. 
\begin{theorem}\label{thm:main}
For every $p\in[0,1]$ there is a deterministic $v=v(p)\in\R^d$ such that
\[
 P^{\bs{\omega}}\!\left(\lim_{n\to\infty}\frac{X_n}{n}=v\right)=1
 \quad\text{for $\Q$-almost every }\bs{\omega}.
\]
\end{theorem}
\begin{remark}
Or equivalently, that there is a deterministic $v$ such that  
$\lim n^{-1}X_n=v$, $\P$-almost surely.
\end{remark}
\begin{remark} 
We will show that there is a deterministic $\theta=\theta(p)\in[0,1]$ such that
\[
 v=\theta m_1+(1-\theta)m_2.
\]
Obviously if $m_1=m_2$ then $v=m_1=m_2$. Moreover, if $m_1$ and $m_2$ are linearly independent, we obtain \textit{ballisticity} (i.e. $v\neq 0$).\footnote{The OpenAI release referred to earlier also claimed a proof of ballisticity under uniform ellipticity and directional transience.}
\end{remark}

In fact we can upgrade Theorem \ref{thm:main} to arbitrary i.i.d.~mixtures of $\mu_1, \mu_2$ as follows.
\begin{theorem}\label{thm:mixture}
Let $(Y_x)_{x\in\Z^d}$ be i.i.d.~random variables in $[0,1]$, and let
\[
 \omega_x=Y_x\mu_1+(1-Y_x)\mu_2,
\]
For the nearest-neighbour RWRE in this environment, there is a deterministic limiting velocity. 
\end{theorem}
\begin{remark}
Theorem \ref{thm:mixture} is equivalent to saying that there is a deterministic velocity for any  i.i.d.~environment law supported on an affine line in the transition-probability simplex (take the endpoints of the line as $\mu_1$ and $\mu_2$ respectively).
\end{remark}
In the case of $d=2$, the statement of Theorem \ref{thm:mixture} is already known   \cite[Theorem 1.5]{HS}.  The generalisation to dimensions $d\ge 3$ is new.  Our proof will use the well-known existence of the (possibly random) limit \cite[Theorem 1.4]{HS}.
\begin{theorem}[\cite{HS}]
\label{thm:V}
If $\bs{\omega}$ is an i.i.d.~environment then there is a $V=V(\bs{\omega})\in\R^d$ such that
\[
 P^{\bs{\omega}}\!\left(\lim_{n\to\infty}\frac{X_n}{n}=V\right)=1
 \quad\text{for $\Q$-almost every }\bs{\omega}.
\]
\end{theorem}
The cited results from \cite{HS}  do not assume any ellipticity, but they are adaptions of  results of Sznitman and Zerner \cite{SZ}, Zerner \cite{Zern02,Zern07} and Zerner and Merkl \cite{ZM01} that do.
Berger \cite{Berger} has shown in the setting of uniform ellipticity (where there exists $\varepsilon>0$ such that $\omega_o(e)>\varepsilon$ for every $e\in \mc{E}$ almost surely) that the support of $V$ contains at most 1 point different from 0.

\subsection{Idea of the argument}
\label{sec:idea}
In this section we present  the general approach of the argument of Theorem \ref{thm:main},  which is an adaption of a coupling argument of \cite{HS}. 

\begin{proof}[Idea of the proof of Theorem \ref{thm:main}]
Let $\bs{\eta}=(\eta_x)_{x\in\Z^d}$ be i.i.d.~labels, with
\[
 \Prob(\eta_x=1)=p,\qquad \Prob(\eta_x=2)=1-p.
\]
Set $\omega_x=\mu_{\eta_x}$, i.e.~
\[\omega_x=\begin{cases}
\mu_1, & \text{ if }\eta_x=1\\
\mu_2, & \text{ if }\eta_x=2.
\end{cases}\]
Henceforth we will refer to $\bs{\eta}$ as the \emph{scenery}.   
A depiction of such a  scenery is given in Figure \ref{fig:environment} in the case $d=2$.
\begin{figure}
\begin{center}
\begin{tikzpicture}[ scale=0.8, every node/.style={font=\small} ] 
\def\p{0.6} 
\pgfmathsetseed{12345}
\draw[gray!50, thin] (-3,-3) grid (3,3); 
\draw (0,0) circle (6pt);
\foreach \i in {-3,...,3} { \foreach \j in {-3,...,3} { \pgfmathrandominteger{\r}{0}{9999} \pgfmathparse{ifthenelse(\r < 10000*\p,1,2)} \pgfmathprintnumberto[precision=0]{\pgfmathresult}{\valueatvertex} \node at (\i,\j) {\valueatvertex}; } }
 \end{tikzpicture}
 \end{center}
\caption{An example of a subset of the scenery $\bs{\eta}$ in the case $d=2$.  The origin is in the centre and is circled.}
\label{fig:environment}
\end{figure}
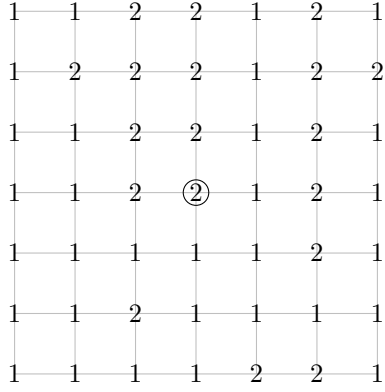
Thus, for fixed $\bs{\eta}$ transition probabilities of the walk, started at zero, are
\[
 P^{\bs \eta}(X_{n+1}=x+e\mid X_n=x)=\mu_{\eta_x}(e).
\]

Independently of $\bs{\eta}$, sample two independent \emph{bi-infinite}\footnote{in \cite{HS} they are only indexed by $\mathbb{N}=\{1,2,\dots\}$} sequences (which we will call the \emph{streams})
\[\bs{\xi}^{\sone}=
 (\xi^{\sone}_k)_{k\in\Z},\qquad \bs{\xi}^{\stwo}=(\xi^{\stwo}_k)_{k\in\Z},
\]
with $\xi^{\sr}_k$ i.i.d.~with law $\mu_r$.  In other words each $\xi_k^{\sr}$ takes values in $\mc{E}$.   
The streams are depicted in table form as follows (an example is given in Table \ref{tab:moves}):

\medskip

\begin{tabular}{c|cccc|ccccc}
index & $\dots$ & $-3$ & $-2$ & $-1$ & $0$ & $1$ & $2$ & $3$ & $\dots$\\
\hline
&&&&&&&&&\\
$\mu_1$ moves & $\dots $ & $\xi^{\sone}_{-3}$ & $\xi^{\sone}_{-2}$ & $\xi^{\sone}_{-1}$ & $\xi^{\sone}_{0}$ & $\xi^{\sone}_{1}$ & $\xi^{\sone}_{2}$ & $\xi^{\sone}_{3}$ & $\dots$\\
&&&&&&&&&\\
 $\mu_2$ moves & $\dots $ & $\xi^{\stwo}_{-3}$ & $\xi^{\stwo}_{-2}$ & $\xi^{\stwo}_{-1}$ & $\xi^{\stwo}_{0}$ & $\xi^{\stwo}_{1}$ & $\xi^{\stwo}_{2}$ & $\xi^{\stwo}_{3}$ & $\dots$
\end{tabular}

\medskip

\begin{table}
\begin{tabular}{c|cccc|ccccc}
index & $\dots$ & $-3$ & $-2$ & $-1$ & $0$ & $1$ & $2$ & $3$ & $\dots$\\
\hline
$\mu_1$ moves & $\dots $ & $e_2$ & $e_2$ & $e_1$ & $-e_1$ & $e_1$ & $-e_1$ & $e_2$ & $\dots$\\
 $\mu_2$ moves & $\dots $ & $-e_1$ & $-e_2$ & $e_1$ & $e_1$ & $-e_2$ & $-e_1$ & $e_1$ & $\dots$
\end{tabular}
\caption{An example of a realisation of two streams  of ``moves'' $\bs{\xi}^{\sone}$, $\bs{\xi}^{\stwo}$ when $d=2$.}
\label{tab:moves}
\end{table}
Let $\bs{\zeta}=(\bs{\eta},\bs{\xi}^{\sone}, \bs{\xi}^{\stwo})$ 
and let 
$\mathbb \P$ denote the joint product law on $\mathcal{G}=\sigma(\bs{\zeta})$.  
One can view the random walk in a 2-valued i.i.d.~environment on this space as follows (see Lemma \ref{lem:law} below.)  Starting from the origin $o=(0,\dots,0)\in \Z^d$, (i.e.~$X_0=o$), at time $n$ look at the environment at the current location, $X_n$.  If it is type $r\in\{1,2\}$, and if we have previously departed from sites of type $r$, $M^{\sr}_n=k$ times before, we use $\xi_k^{\sr}$ to determine our next move, i.e.~$X_{n+1}=X_n+\xi_k^{\sr}$ on this event.  After making this move we have departed from a type $r$ site one additional time, so on the next departure from such a site we will use $\xi_{k+1}^{\sr}$.  

\begin{example}
\label{exa:1}
\START
We illustrate this with a concrete example in two dimensions, using Figure \ref{fig:environment} and Table \ref{tab:moves}.  Starting from $X_0=o=(0,0)\in \Z^2$, we see in Figure \ref{fig:environment} that the environment at $X_0$ is type 2.  Our first move is then determined by $\xi^{\stwo}_0$, i.e.~the random variable in row 2 (corresponding to type 2) of index 0, which is the first one to the right of the vertical line in the table (corresponding to the first departure from a vertex of type 2).  In the example of Table \ref{tab:moves}, the value of $\xi^{\stwo}_0=\rightarrow=e_1\in \mc{E}$.  Thus set $X_1=X_0+e_1=(1,0)$.   Now look at the environment at $X_1=(1,0)$.  In Figure \ref{fig:environment} this is type 1.  So we now use the move $\xi^{\sone}_0$, i.e.~the first move to the right of the vertical line in the table.  In the example of Table \ref{tab:moves} this is $\xi^{\sone}_0=\leftarrow=-e_1$, so $X_2=X_1-e_1=(0,0)$.  The environment at $X_2=(0,0)$ is type 2 (a type that we have already departed from once before) so we now use $\xi^{\stwo}_1$, which in Table \ref{tab:moves} is $\downarrow=-e_2$ so $X_3=(0,-1)$.  As an exercise, check that in this example the next two moves are to $X_4=(1,-1)$ and $X_5=(0,-1)$. \END 
\end{example}
The following is elementary.
\begin{lemma}\label{lem:law}
Conditional on $\bs{\eta}$ (equivalently $\bs{\omega}$), $\bs{X}=(X_n)_{n\ge 0}$ has the prescribed law $P^{\bs{\omega}}$.
\end{lemma}

From Theorem \ref{thm:V} the limit $V=\lim_{n \to \infty}n^{-1}X_n$ exists.  But $X_n$ is composed of $M^{\sone}_{n}$ type $1$ steps and $n-M^{\sone}_{n}$ type $2$, where as above $M^{\sone}_{n}$ is the number of times that the walk has departed from type 1 vertices after $n$ steps.  If the means $m_1$ and $m_2$ are the same it is then easy to see that $V=m_1=m_2$ is deterministic.  So the interesting setting is where $u=m_1-m_2\ne o$.  In this case from the above argument we can conclude that $n^{-1}M^{\sone}_{n}\to \Theta$ almost surely, and that 
\[V=\Theta m_1+(1-\Theta)m_2, \qquad \Theta=\frac{u \cdot (V-m_2)}{\|u\|^2}.\]
Thus it is sufficient to show that $\Theta$ is deterministic.  

The proof constructs a certain antidiagonal transformation $T$ on $\bs{\zeta}$ (in particular it shifts the $\bs{\xi}^{\sone}$ sequence one step forward and the $\bs{\xi}^{\stwo}$ sequence one step backward) that is ergodic under $\mathbb{P}$.  It will have the property that $M^{\sone}_n(\bs{\zeta})\le 1+ M^{\sone}_n(T\bs{\zeta})$ for all $n\ge 0$ and hence that 
\[\Theta\le \Theta\circ T, \qquad \mathbb{P}-\text{almost surely}.\]  At the same time $T$ being measure preserving gives (with $\E$ denoting expectation with respect to $\P$)
\[\E[\Theta]=\E[\Theta\circ T].\]  Since $\Theta\in [0,1]$ this implies that $\Theta\circ T=\Theta$ almost surely, so ergodicity implies that $\Theta$  is almost surely constant, completing the proof.
\end{proof}

The majority of the remainder of the paper gives details of the above argument.

\section{The probability space}
\label{sec:probspace}
We start by presenting a variant of the coupling argument of \cite{HS}.    Recall that $\bs{\zeta}=(\bs{\eta},\bs{\xi}^{\sone}, \bs{\xi}^{\stwo})$ 
and that  
$\mathbb P$ denotes the joint product law on $\mathcal{G}=\sigma(\bs{\zeta})$.  Let $\bs{\xi}=(\bs{\xi}^{\sone}, \bs{\xi}^{\stwo})$.

Define bi-infinite random walks $\bs{S}^{\sr}=(S^{\sr}_n)_{n \in \Z}$ with step distribution $\mu_r$ from the $\bs{\xi}^{\sr}$, i.e.
\[
 S^{\sr}_0=0,\qquad S^{\sr}_{k+1}-S^{\sr}_k=\xi^{\sr}_k, \qquad k\in\Z.
\]
In particular
\[
 S^{\sr}_n=\sum_{k=0}^{n-1}\xi^{\sr}_k,\qquad
 S^{\sr}_{-n}=-\sum_{k=-n}^{-1}\xi^{\sr}_k, \qquad n\ge 1.
\]
Define a random map $\Phi=\Phi_{\bs{\xi}}:\Z^2 \to \Z^d$ by 
\begin{equation}\label{eq:phi}
 \Phi(i,j)=S^{\sone}_i+S^{\stwo}_j,\qquad (i,j)\in\Z^2.
\end{equation}
This map depends only on the streams (and is hence independent of the scenery).

Create a $\mc{G}$-measurable random arrow field by placing an arrow at each count-lattice site $(i,j)$ towards $(i,j)+e_{\eta_{\Phi(i,j)}}$, i.e.
\begin{equation}\label{eq:arrows}
 (i,j)\longmapsto
 \begin{cases}
 (i+1,j),&\text{ if }\eta_{\Phi(i,j)}=1,\\
 (i,j+1),&\text{ if }\eta_{\Phi(i,j)}=2.
 \end{cases}
\end{equation}
This depends on both the scenery and the streams.
\begin{example}
\START Continuing from Example \ref{exa:1}, we have e.g.~that $\Phi(0,0)=o\in \Z^2$ and 
\begin{align}
\Phi(-3,-3)&=-(\sum_{k=-3}^{-1}\xi^{\sone}_k)-(\sum_{\ell=-3}^{-1}\xi^{\stwo}_k)\\
&=-e_1-e_2-e_2-e_1-(-e_2)-(-e_1)=-e_1-e_2=(-1,-1).
\end{align}
Figure \ref{fig:Phi} shows the values of $\Phi(i,j)$ for $(i,j)\in \{-3,-2,\dots, 2, 3\}^2$ as determined from Table \ref{tab:moves}.  It is helpful to start from $\Phi(-3,-3)=(-1,-1)$ and add a corresponding $\mu_1$ (resp.~$\mu_2$) move from Table \ref{tab:moves} if we increase $i$ (resp.~$j$) by 1.  

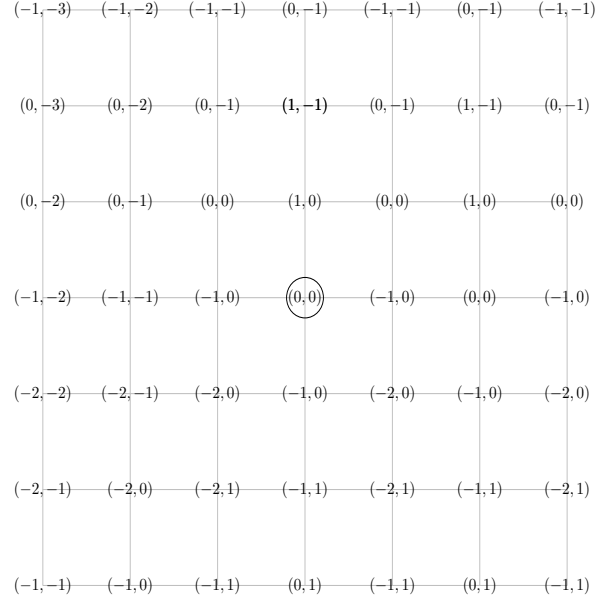
\begin{figure}
\begin{center}
\resizebox{8cm}{8cm}{
\begin{tikzpicture}[ scale=2, every node/.style={font=\small} ] 
\draw[gray!50, thin] (-3,-3) grid (3,3);
\draw (0,0) circle (6pt);
\node at (-3,-3) {$(-1,-1)$};
\node at (-3,-2) {$(-2,-1)$};
\node at (-3,-1) {$(-2,-2)$};
\node at (-2,-3) {$(-1,0)$};
\node at (-1,-3) {$(-1,1)$};
\node at (0,-3) {$(0,1)$};
\node at (1,-3) {$(-1,1)$};
\node at (2,-3) {$(0,1)$};
\node at (3,-3) {$(-1,1)$};

\node at (-2,-2) {$(-2,0)$};
\node at (-1,-2) {$(-2,1)$};
\node at (0,-2) {$(-1,1)$};
\node at (1,-2) {$(-2,1)$};
\node at (2,-2) {$(-1,1)$};
\node at (3,-2) {$(-2,1)$};

\node at (-2,-1) {$(-2,-1)$};
\node at (-1,-1) {$(-2,0)$};
\node at (0,-1) {$(-1,0)$};
\node at (2,-1) {$(-1,0)$};
\node at (3,-1) {$(-2,0)$};

\node at (0,0) {$(0,0)$};
\node at(1,0) {$(-1,0)$};
\node at (2,0) {$(0,0)$};
\node at (3,0) {$(-1,0)$};
\node at(-1,0) {$(-1,0)$};
\node at(-2,0) {$(-1,-1)$};
\node at (-3,0) {$(-1,-2)$};
\node at (0,1) {$(1,0)$};
\node at (0,2) {$(1,-1)$};
\node at (0,3) {$(0,-1)$};
\node at (1,-1) {$(-2,0)$};
\node at (-3,1) {$(0,-2)$};
\node at (-2,1) {$(0,-1)$};
\node at (-1,1) {$(0,0)$};
\node at (1,1) {$(0,0)$};
\node at (2,1) {$(1,0)$};
\node at (3,1) {$(0,0)$};
\node at (-3,2) {$(0,-3)$};
\node at (-3,3) {$(-1,-3)$};
\node at (-2,2) {$(0,-2)$};
\node at (-1,2) {$(0,-1)$};
\node at (0,2) {$(1,-1)$};
\node at (1,2) {$(0,-1)$};
\node at (2,2) {$(1,-1)$};
\node at (3,2) {$(0,-1)$};
\node at (-2,3) {$(-1,-2)$};
\node at (-1,3) {$(-1,-1)$};
\node at (1,3) {$(-1,-1)$};
\node at (2,3) {$(0,-1)$};
\node at (3,3) {$(-1,-1)$};

\end{tikzpicture}
}
\end{center}
\caption{The values of $\Phi(i,j)$ for $-3\le i,j\le 3$ determined from Table \ref{tab:moves}.}
\label{fig:Phi}
\end{figure}
We can now infer the arrows at such $i,j$ using the information in Figures \ref{fig:environment} and \ref{fig:Phi}.  Since $\eta_{(0,0)}=2$ the arrow at all $(i,j)$ for which $\Phi(i,j)=(0,0)$ will be $\uparrow$.  The corresponding arrow system is given in Figure \ref{fig:arrows1}.
\END
\end{example}

\begin{figure}
\begin{center}
\resizebox{8cm}{8cm}{
\begin{tikzpicture}[ scale=2, every node/.style={font=\small} ] 
\draw[gray!50, thin] (-3,-3) grid (3,3);
\draw (0,0) circle (6pt);
\draw[dashed] (-3.5,3.5)--(3.5,-3.5);
\draw[->,ultra thick]  (-3,-3)--(-2.5,-3);
\draw[->,ultra thick]  (-3,-2)--(-2.5,-2); 
\draw[->,ultra thick] (-3,-1)--(-2.5,-1);
\draw[->,ultra thick] (-2,-3)--(-2,-2.5);
\draw[->,ultra thick] (-1,-3)--(-1,-2.5);
\draw[->,ultra thick] (0,-3)--(0,-2.5);
\draw[->,ultra thick] (1,-3)--(1,-2.5);
\draw[->,ultra thick] (2,-3)--(2,-2.5);
\draw[->,ultra thick] (3,-3)--(3,-3+0.5);
\draw[->,ultra thick] (-2,-2)--(-2+0.5,-2);
\draw[->,ultra thick] (-1,-2)--(-1+0.5,-2);
\draw[->,ultra thick] (0,-2)--(0,-2+0.5);
\draw[->,ultra thick](1,-2)--(1+0.5,-2);
\draw[->,ultra thick] (2,-2)--(2,-2+0.5);
\draw[->,ultra thick] (3,-2)--(3+0.5,-2);

\draw[->,ultra thick] (-2,-1)--(-2+0.5,-1);
\draw[->,ultra thick] (-1,-1)--(-1+0.5,-1);
\draw[->,ultra thick] (0,-1)--(0,-1+0.5);
\draw[->,ultra thick] (1,-1)--(1+0.5,-1);
\draw[->,ultra thick] (2,-1)--(2,-1+0.5);
\draw[->,ultra thick] (3,-1)--(3+0.5,-1);

\draw[->,ultra thick] (0,0)--(0,0+0.5);
\draw[->,ultra thick] (1,0)--(1,0+0.5);
\draw[->,ultra thick] (2,0)--(2,0+0.5);
\draw[->,ultra thick] (3,0)--(3,0+0.5);

\draw[->,ultra thick] (-1,0)--(-1,0+0.5);
\draw[->,ultra thick] (-2,0)--(-2+0.5,0);
\draw[->,ultra thick] (-3,0)--(-3,0+0.5);
\draw[->,ultra thick] (0,1)--(0+0.5,1);
\draw[->,ultra thick] (0,2)--(0+0.5,2);
\draw[->,ultra thick] (0,3)--(0+0.5,3);
\draw[->,ultra thick] (1,-1)--(1+0.5,-1);
\draw[->,ultra thick] (-3,1)--(-3+0.5,1);
\draw[->,ultra thick] (-2,1)--(-2+0.5,1);
\draw[->,ultra thick] (-1,1)--(-1,1+0.5);
\draw[->,ultra thick] (1,1)--(1,1+0.5);
\draw[->,ultra thick] (2,1)--(2+0.5,1);
\draw[->,ultra thick] (3,1)--(3,1+0.5);
\draw[->,ultra thick] (-3,2)--(-3+0.5,2);
\draw[->,ultra thick] (-3,3)--(-3+0.5,3) ;
\draw[->,ultra thick] (-2,2)--(-2+0.5,2) ;
\draw[->,ultra thick] (-1,2)--(-1+0.5,2) ;
\draw[->,ultra thick] (0,2)--(0+0.5,2) ;
\draw[->,ultra thick] (1,2)--(1+0.5,2) ;
\draw[->,ultra thick] (2,2)--(2+0.5,2) ;
\draw[->,ultra thick] (3,2)--(3+0.5,2) ;
\draw[->,ultra thick] (-2,3)--(-2,3+0.5) ;
\draw[->,ultra thick] (-1,3)--(-1+0.5,3) ;
\draw[->,ultra thick] (1,3)--(1+0.5,3);
\draw[->,ultra thick] (2,3)--(2+0.5,3);
\draw[->,ultra thick] (3,3)--(3+0.5,3);

\end{tikzpicture}
}
\end{center}
\caption{A portion of the arrow field for the example corresponding to Figure \ref{fig:environment} and Table \ref{tab:moves}.}
\label{fig:arrows1}
\end{figure}
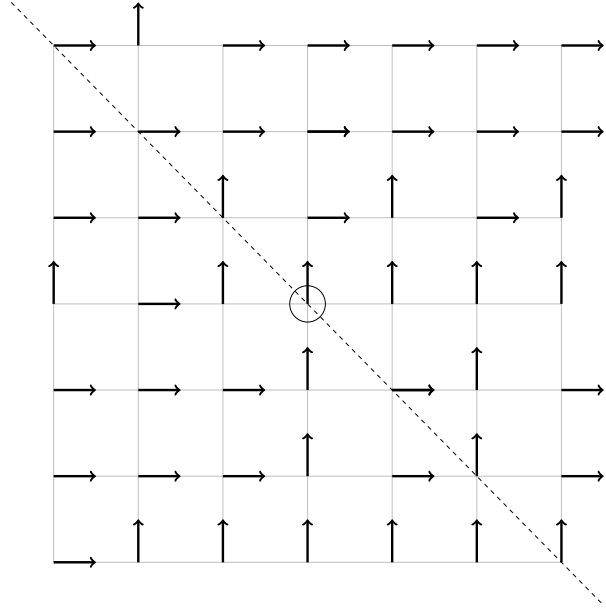

For $a\in\Z^2$, let $\bs{C}^a=(C^a_n)_{n\ge 0}$ (here $\bs{C}^a=\bs{C}^a(\bs{\zeta})$) be the path in the arrow field started at $a$ obtained by following arrows.  This depends on both the scenery and the streams. Write
\[
 C^o_n=(M^{\sone}_n,M^{\stwo}_n),\qquad M^{\sone}_n+M^{\stwo}_n=n.\]
Let 
\begin{equation}\label{eq:walk}
 X_n=\Phi(C^o_n)=S^{\sone}_{M^{\sone}_n}+S^{\stwo}_{M^{\stwo}_n}.
\end{equation}
Equation \ref{eq:walk} is just a rewording of the construction of the walk as in Section \ref{sec:idea}.  To see this note first that the walk starts at $\Phi(0,0)=o$.  If $C^o_n=(k,n-k)$ then the walker is at position $x=\Phi(C^o_n)=S^{\sone}_{k}+S^{\stwo}_{n-k}$ and has used the steps $\xi^{\sone}_0,\dots,\xi^{\sone}_{k-1}$ and $\xi^{\stwo}_0,\dots,\xi^{\stwo}_{n-k-1}$.  The scenery at $x$ is $r=\eta_x$, so the arrow at $(k,n-k)$ will be $\rightarrow$ if and only if $r=1$ (i.e.~iff the environment at $x$ is of type 1).  Thus $C^o_{n+1}=(k,n-k)+e_{r}$ and so the $r$-th component of $\bs{C}^o$ grows by 1.  This means that the increment of $X_n$ is $\xi^{\sone}_{k}$ if $r=1$ or $\xi^{\stwo}_{n-k}$ otherwise, and this increment is independent of the past, with the distribution $\mu_r$ as required.
As described in Lemma \ref{lem:law}, conditional on the scenery $\bs{\eta}$, this random walk has the correct law $P^{\bs{\eta}}$ (and $M_n^{\sr}$ is the number of times that the walker has selected a step from $\mu_r$ after $n$ steps of the walk).   

The following is implicit in \cite[pages 17-18]{HS} (in particular Corollary 4.3 therein), but we include a proof here for completeness.
\begin{lemma}\label{lem:slopes}
Almost surely,
\begin{equation}\label{eq:drift}
 \frac{X_n}{n}-\Big(m_2+(m_1-m_2)\frac{M_n^{\sone}}{n}\Big)\to 0.
\end{equation}
Moreover if $m_1\ne m_2$  then
 $M_n^{\sone}/n$ converges almost surely to a random variable $\Theta\in[0,1]$.\footnote{In fact it also holds for $m_1=m_2$ (see \cite[Corollary 4.3]{HS}) but we will not need this} 
\end{lemma}
\begin{proof}
For each stream, the strong law implies
\begin{equation}
 \frac1n\cdot \max_{0\leq k\leq n}\norm{S^{\sss (r)}_k-km_r}\longrightarrow0, \qquad \text{almost surely}.\label{fish1}
\end{equation}
Indeed, for any $\ee>0$, all sufficiently large $k$ have error at most $\ee k$, while the finitely many earlier errors divided by $n$ vanish. 
Now 
\[\frac{X_n}{n}=\frac{S^{\sone}_{M_n^{\sone}}-M_n^{\sone}m_1}{n}+\frac{S^{\stwo}_{M_n^{\stwo}}-M_n^{\stwo}m_2}{n}+\frac{M_n^{\sone}}{n}m_1+\frac{M_n^{\stwo}}{n}m_2.\]
Using \eqref{fish1} and then the fact that $M_n^{\sone}+M_n^{\stwo}=n$ gives the almost surely limit  \eqref{eq:drift}. 

When $m_1\ne m_2$, put $\delta=m_1-m_2$. If $V=\lim X_n/n$, then taking the dot product of  \eqref{eq:drift} with $\delta/\|\delta\|^2$ gives
\[
 \frac{M_n^{\sone}}{n}
 =\frac{\delta\cdot(X_n/n-m_2)}{\norm{\delta}^2}+o(1)
 \longrightarrow
 \frac{\delta\cdot(V-m_2)}{\norm{\delta}^2}=:\Theta.
\]
\end{proof}

It remains to show that $\Theta$ is deterministic. From now on assume $m_1\ne m_2$.

\section{The antidiagonal transformation}
Recall that $\bs{\zeta}=(\bs{\eta},\bs{\xi}^{\sone}, \bs{\xi}^{\stwo})$.  Define $T\bs{\zeta}=\bs{\zeta}'=(\bs{\eta}',\bs{\xi}^{\sone\prime},\bs{\xi}^{\stwo\prime})$ by
\begin{align}
 \xi^{\sone\prime}_k&=\xi^{\sone}_{k+1},&
 \xi^{\stwo\prime}_k&=\xi^{\stwo}_{k-1},\label{eq:Tstreams}\\
 D&=S^{\sone}_1+S^{\stwo}_{-1}=\xi^{\sone}_0-\xi^{\stwo}_{-1},&
 \eta'_x&=\eta_{x+D}.\label{eq:Teta}
\end{align}
Thus one stream is shifted forward, the other backward, and the scenery is recentered by the corresponding physical displacement.

\begin{lemma}\label{lem:preserving}
$T$ preserves $\mathbb P$.
\end{lemma}
\begin{proof}
The shifted streams have their original joint product law. Conditional on both entire streams, $D$ is fixed. Since the scenery is independent of the streams and its law is invariant under every lattice translation, $(\eta_{x+D})_x$ has its original law conditional on the streams. In particular, this recentred scenery remains independent of the shifted streams.
\end{proof}

\paragraph{Shift covariance.}
The cumulative sums $S_n^{\sr\prime}=\sum_{k=0}^{n-1}\xi^{\sr\prime}_k$ and $S_{-n}^{\sr\prime}=-\sum_{k=-n}^{-1}\xi^{\sr\prime}_k$ for $n \in \mathbb{N}$ (and with $S^{\sr\prime}_0=0$) associated with $\bs{\zeta}'=T\bs{\zeta}$ satisfy
\[
 S^{\sone\prime}_i=S^{\sone}_{i+1}-S^{\sone}_1,\qquad
 S^{\stwo\prime}_j=S^{\stwo}_{j-1}-S^{\stwo}_{-1}.
\]
Consequently,
\begin{align}
\label{eq:covariance1}
 \Phi_{\bs{\xi}'}(i,j)&=S_i^{\sone\prime}+S_j^{\stwo\prime}
 =S^{\sone}_{i+1}-S^{\sone}_1+S^{\stwo}_{j-1}-S^{\stwo}_{-1}
 =  \Phi_{\bs{\xi}}(i+1,j-1)-D,
\end{align}
which depends only on the shifted streams (so is independent of the environment and shifted environment).  Therefore
\begin{equation}
 \eta'_{\Phi_{\bs{\xi}'}(i,j)}=\eta'_{\Phi_{\bs{\xi}}(i+1,j-1)-D}=\eta_{\Phi_{\bs{\xi}}(i+1,j-1)}.\label{covariance2}
\end{equation}
Let $b=(1,-1)$. Equations  \eqref{eq:covariance1} and \eqref{covariance2} give the exact path identity
\begin{equation}\label{eq:pathcov}
 C^b_n(\bs{\zeta})=b+C^o_n(\bs{\zeta}').
\end{equation}
Indeed, on the left hand side, the arrow at $(i,j)+b$ is to $(i+1,j)+b$ if and only if $\eta_{\Phi_{\bs{\xi}}(i+1,j-1)}=1$.  On the right hand side the arrow at $(i,j)$ is to $(i+1,j)$ if and only if $\eta'_{\Phi_{\bs{\xi}'}(i,j)}=1$, i.e.~if and only if $\eta_{\Phi_{\bs{\xi}}(i+1,j-1)}=1$.

Recall the following standard notions.
\begin{definition}
\START $T$ is \emph{mixing} if for every pair of bounded measurable real-valued functions $f, g$,
\[\lim_{n \to \infty}\E[f\cdot(g\circ T^{(n)})]=\E[f]\E[g].\]
It is \emph{ergodic} if for any r.v. $Z$, $Z\circ T=Z$ implies $Z$ is a.s. constant.
\END
\end{definition}
It is well known that mixing implies ergodicity, see e.g.~\cite[Theorem 1.9]{Walters}.  
\begin{lemma}\label{lem:mixing}
If $m_1\ne m_2$, the transformation $T$ is mixing, and hence ergodic, under $\mathbb P$.
\end{lemma}
\begin{proof}
Iterating \eqref{eq:Tstreams}--\eqref{eq:Teta} gives $T^{(n)}\bs{\zeta}=\big((\bs{\eta})^{(n)},(\bs{\xi}^{\sone})^{(n)},(\bs{\xi}^{\stwo})^{(n)}\big)$, where
\begin{align*}
 (\xi^{\sone})^{(n)}_k&=\xi^{\sone}_{k+n},&
 (\xi^{\stwo})^{(n)}_k&=\xi^{\stwo}_{k-n},&
 \eta^{(n)}_x&=\eta_{x+D_n},\\
 D_n&=S^{\sone}_n+S^{\stwo}_{-n}.
\end{align*}
By the strong law,
\begin{equation}\label{eq:displacement}
 \frac{D_n}{n}\longrightarrow m_1-m_2\ne0
 \quad\mathbb P\text{-almost surely}.
\end{equation}
In particular, $\norm{D_n}\to\infty$.

Let $f,g$ be measurable functions that depend on only finitely many coordinates of $\bs{\zeta}$. Suppose their $\bs{\eta}$  coordinates lie in finite sets $A,B\subset\Z^d$, respectively. For some deterministic $n_0$, and all $n\ge n_0$, the $(\bs{\xi}^{\sone}, \bs{\xi}^{\stwo})$  coordinates used by $f$ and by $g\circ T^{(n)}$ are disjoint: the two stream windows of the latter having shifted by $+n$ and $-n$, respectively.

Condition on the entire streams. The scenery coordinates of $\bs{\eta}$ used by $f$ and by $g\circ T^{(n)}$ are $A$ and $B+D_n$. Whenever these sets are disjoint, their independence (and the fact that $D_n$ depends only on the streams) gives the factorization of the conditional expectation over scenery. To make the dependency bookkeeping explicit, let
\[
 \bar f(\bs{\xi}^{\sone},\bs{\xi}^{\stwo})=\E[f(\bs{\eta},\bs{\xi}^{\sone},\bs{\xi}^{\stwo})\mid \bs{\xi}^{\sone},\bs{\xi}^{\stwo}],
\]
and define $\bar g$ analogously. Write $T_0$ for the transformation of the two streams alone. The invariance of the scenery law under translation gives
\[
 \E[g\circ T^{(n)}\mid\bs{\xi}^{\sone},\bs{\xi}^{\stwo}]
 =\bar g\circ T_0^{(n)} \text{, a.s.}
\]
Therefore on $\{A\cap(B+D_n)=\varnothing\}$ we have that 
\[
\E[f\cdot (g\circ T^{(n)})\mid\bs{\xi}^{\sone},\bs{\xi}^{\stwo} ]
=\E[f\mid\bs{\xi}^{\sone},\bs{\xi}^{\stwo} ]\cdot \E[g\circ T^{(n)}\mid\bs{\xi}^{\sone},\bs{\xi}^{\stwo} ]
=\bar f\cdot  (\bar g\circ T_0^{(n)})\text{, a.s.}
\]
It follows that
\begin{equation}
\left|\E[f\cdot (g\circ T^{(n)})]
       -\E[\bar f\cdot (\bar g\circ T_0^{(n)})]\right|\leq
 2\norm{f}_\infty\norm{g}_\infty\,
 \mathbb P(A\cap(B+D_n)\ne\varnothing),
 \label{eq:mixingbound}
\end{equation}
whenever $n\ge n_0$. The probability on the right tends to zero by \eqref{eq:displacement}, since $A$ and $B$ are finite. For $n\ge n_0$, independence of the disjoint stream-coordinate windows also yields
\[
 \E[\bar f(\bar g\circ T_0^{(n)})]
 = \E[\bar f]\cdot \E[\bar g\circ T_0^{(n)}]
 =\E[f]\E[g].
\]
Thus mixing holds for bounded cylinder functions. These are dense in $L^2(\mathbb P)$. Measure preservation and the Cauchy--Schwarz inequality now extend the conclusion to all bounded measurable functions, proving mixing.
\end{proof}

\begin{lemma}\label{lem:order}
For every realization of $\bs{\zeta}$ and every $n\geq0$,
\begin{equation}\label{eq:order}
 M_n^{\sone}(\bs{\zeta})\leq 1+M^{\sone}_n(\bs{\zeta}').
\end{equation}
\end{lemma}
\begin{proof}
First note that for $n=0$ the inequality is simply $0\le 1+0$ so is true.  
From \eqref{eq:pathcov} we have that 
\begin{equation}
\label{eq:pathcov2}
 C^o_n(\bs{\zeta}')=C^b_n(\bs{\zeta})-b.
\end{equation}
Let $K_n$ be the first coordinate of $C^b_n(\bs{\zeta})$.  Then  $M^{\sone}_n(\bs{\zeta}')=K_n-1$ by \eqref{eq:pathcov2}.  Since in the arrow system each arrow points either north or east we have that $(C^x_{n+1}(\bs{\zeta})-C^x_{n}(\bs{\zeta}))\cdot (1,1)=1$ for any $x$.  So for any $x$ of the form $x=(k,-k)$ for $k \in \Z$ (whence $C^x_{0}(\bs{\zeta})\cdot (1,1)=0$) we have that  $C^x_{n}(\bs{\zeta})\cdot (1,1)=n$.   Both $b=(1,-1)$ and $o=(0,0)$ are of this form.  Let $\tau=\inf\{n:K_n=M_n^{\sone}(\bs{\zeta})\}$, which is the first time that the first coordinates of the two paths in the $\bs{\zeta}$ arrow system match.   
 Since the sums of coordinates of $C^b_\tau(\bs{\zeta})$ and $C^o_\tau(\bs{\zeta})$ are both equal to $\tau$, also the second coordinates are the same at this time, so $C^o_{\tau}(\bs{\zeta})=C^b_{\tau}(\bs{\zeta})$.  From time $\tau$ both paths are identical since they evolve in the same arrow system.  It follows that for all $n\ge \tau$, 
 \[M_n^{\sone}(\bs{\zeta}')=M_n^{\sone}(\bs{\zeta})-1.\]
 Now $K_0-M^{\sone}_0(\bs{\zeta})=1$.  The difference $K_{n+1}-M^{\sone}_{n+1}(\bs{\zeta})$ can only change by $0, -1,+1$ at each time step, so for all $n<\tau$ we have $K_n\ge M^{\sone}_n(\bs{\zeta})+1$.  In other words 
 \[M_n^{\sone}(\bs{\zeta}')=K_n-1\ge M^{\sone}_n(\bs{\zeta}).\]
\end{proof}
We can now complete the proof of Theorem \ref{thm:main}.
\begin{proof}[Proof of Theorem \ref{thm:main}]
The equal-mean case follows trivially from Lemma \ref{lem:slopes}. Otherwise $M_n^{\sone}(\bs{\zeta})/n\to\Theta$, $\P$-almost surely. Since $T$ is measure preserving, 
$M_n^{\sone}(\bs{\zeta}')/n\to\Theta'=\Theta\circ T$ almost surely.  Divide \eqref{eq:order} by $n$ and let $n\to\infty$ to obtain
\[
 \Theta\leq\Theta\circ T,\quad\mathbb P\text{-almost surely}.
\]
But $0\leq\Theta\leq1$ and measure preservation gives
\[
 \E[\Theta]=\E[\Theta\circ T].
\]
The (nonnegative) difference must therefore vanish, i.e.
\[
 \Theta=\Theta\circ T\quad\mathbb P\text{-almost surely}.
\]
By ergodicity from Lemma \ref{lem:mixing}, $\Theta$ is almost surely a deterministic constant $\theta$. Substitution in \eqref{eq:drift} yields
\[
 \frac{X_n}{n}\longrightarrow \theta m_1+(1-\theta)m_2, \qquad \mathbb{P}-\text{almost surely.}
\]
\end{proof}

\section{Extension to Theorem \ref{thm:mixture}}
\label{sec:mixture}
Let $\bs{Y}=(Y_x)_{x\in \Z^d}$ be i.i.d.~taking values in $[0,1]$ as in the statement of Theorem \ref{thm:mixture}.  As in the theorem set 
\begin{equation}
\omega_x=Y_x\mu_1+(1-Y_x)\mu_2.\label{omega_mix_def}\end{equation}
On every departure from $x\in \Z^d$ we sample a step from either $\mu_1$ or $\mu_2$, with the probability of the former (given $Y_x$) being $Y_x$.   
Theorem \ref{thm:main} concerns the special case where  each $Y_x$ is supported on $\{0,1\}$ (with $\eta_x=2-Y_x$).  
\begin{proof}[Proof of  Theorem \ref{thm:mixture}]
Use the same bi-infinite  streams with laws $\mu_1$ and $\mu_2$ respectively.  Independently, sample i.i.d.~uniforms $\bs{U}=(U_{i,j})_{(i,j)\in\Z^2}$.  Now $\bs{\zeta}=(\bs{Y},\bs{U}, \bs{\xi}^{\sone}, \bs{\xi}^{\stwo} )$ and again  
$\mathbb P$ denotes the joint product law on $\mathcal{G}=\sigma(\bs{\zeta})$.  Define the $S^{\sr}_i$ variables and the  physical position $\Phi(i,j)$ associated with a count vertex $(i,j)\in \Z^2$ exactly as before (see \eqref{eq:phi}) depending only on the streams.  The arrow field  depends on $\bs{\zeta}$ as follows (this replaces \eqref{eq:arrows}):
\[
(i,j)\mapsto \begin{cases}
(i+1,j), & \text{ if } U_{i,j}<Y_{\Phi(i,j)},\\
(i,j+1), & \text{ otherwise}.
\end{cases}
\]
The environment $\bs{\omega}$  is i.i.d.~because the $Y_x$ are. Define $\bs{C}^a(\bs{\zeta})$, $M_n^{\sone}(\bs{\zeta})$, $M_n^{\stwo}(\bs{\zeta})$ and $X_n=X_n(\bs{\zeta})$ as before (see \eqref{eq:walk}).  Again, conditional on $\bs{Y}$ (or equivalently $\bs{\omega}$) $\bs{X}=(X_n(\bs{\zeta}))_{n \ge 0}$ has the correct law of the random walk in environment $\bs{\omega}$.  Moreover, $M_n^{\sone}$ is the number of $\mu_1$ steps  (i.e.~$\bs{\xi}^{\sone}$ variables) used up to time $n$.    

Because the environment $\bs{\omega}$ is i.i.d., Theorem \ref{thm:V} applies and Lemma \ref{lem:slopes} still holds (with identical proof).  This handles the case of $m_1=m_2$ as before (i.e.~$X_n/n\to m_2=m_1$, a.s.). So we need only handle the case $m_1\ne m_2$, in which we again know that $\Theta=\lim  \frac{M_n^{\sone}}{n}$ exists. Therefore what we need prove is that $\Theta$ is a.s. constant.

Define $T:\bs{\zeta}\to \bs{\zeta}'=(\bs{Y}',\bs{U}',\bs{\xi}^{\sone \prime},\bs{\xi}^{\stwo\prime })$ by
\begin{align}
 \xi^{\sone\prime}_k&=\xi^{\sone}_{k+1},&
 \xi^{\stwo\prime}_k&=\xi^{\stwo}_{k-1},\label{eq:Tstreams2}\\
 D&=S^{\sone}_1+S^{\stwo}_{-1}=\xi^{\sone}_0-\xi^{\stwo}_{-1},&
 Y'_x&=Y_{x+D}, \qquad U'_{i,j}=U_{i+1,j-1}.\label{eq:TY}
\end{align}
Then $T$ is again measure preserving, and the proof of Lemma \ref{lem:order} applies as before. Moreover if we can show that $T$ is mixing, the remainder of the argument will then apply.

Once again, $T^{(n)}\bs{\zeta}=\big((\bs{Y})^{(n)},\bs{U}^{(n)},(\bs{\xi}^{\sone})^{(n)},(\bs{\xi}^{\stwo})^{(n)}\big)$, where
\begin{align*}
 (\xi^{\sone})^{(n)}_k&=\xi^{\sone}_{k+n},&
 (\xi^{\stwo})^{(n)}_k&=\xi^{\stwo}_{k-n},\\
 D_n&=S^{\sone}_n+S^{\stwo}_{-n},&
 Y^{(n)}_x&=Y_{x+D_n}, \qquad U^{(n)}_{i,j}=U_{i+n,j-n}.
\end{align*}
Moreover $\norm{D_n}$ still $\to \infty$, and $D_n$ depends only on the streams. 

Fix measurable $f$ and $g$ that are bounded and depend only on finitely many coordinates of $\bs{\zeta}$. Then for some large enough deterministic $n_0$, the $(\bs{\xi}^{\sone},\bs{\xi}^{\stwo},\bs{U})$ coordinates of $f$ and $g\circ T^{(n)}$ relied on are once more independent whenever $n\ge n_0$. Let $A$ and $B$ represent the $\bs{Y}$ coordinates used in $f$ and $g$ respectively. It is again the case that conditional on $(\bs{\xi}^{\sone},\bs{\xi}^{\stwo},\bs{U})$, we'll have independence of $f$ and $g\circ T^{(n)}$, at least on the event $\{A\cap(B+D_n)=\varnothing\}$. This implies \eqref{eq:mixingbound}, and the proof now proceeds as before.
\end{proof}
\section*{Acknowledgements}
A version of this argument was obtained by prompting ChatGPT version  GPT-6 Astra Pro. Salisbury's research is supported in part by a grant from NSERC.

\end{document}